\documentclass[a4paper,oneside,12pt] {amsart}

 \usepackage[english]{babel}
\usepackage{amssymb}
\usepackage{color}
\usepackage{amsmath,amscd}
\input{amssym.def}
\input{amssym.tex}

\usepackage{verbatim}   

 \date{\today}

\newcommand{\margp}[1]{}

\newtheorem{theorem}{Theorem}[section]
\newtheorem{lemma}[theorem]{Lemma}
\newtheorem{corollary}[theorem]{Corollary}
\newtheorem{proposition}[theorem]{Proposition}

\newtheorem*{theorem_}{Theorem}

\theoremstyle{definition}
\newtheorem{example}[theorem]{Example}

\newtheorem{remark}[theorem]{Remark}

\numberwithin{equation}{section}

\newcommand{\Ker}{\operatorname{Ker}}

\newcommand{\clos}{\operatorname{clos}}

\newcommand{\rea}{{\rm Re}\,}
\newcommand{\ima}{{\rm Im}\,}

\newcommand\defin {\overset {\text {\rm def} }{=}}
\newcommand\beqn{\begin{equation}}
\newcommand\neqn{\end{equation}}

\newcommand{\A}{\mathcal{A}}
\newcommand{\B}{\mathcal{B}}
\newcommand{\LL}{\mathcal{L}}

\newcommand{\F}{A}

\newcommand{\al}{\alpha}   
\newcommand{\de}{\delta}   
\renewcommand{\kappa}{\varkappa}   
\newcommand{\Ga}{\Gamma}   
\newcommand{\deab}{\varkappa}   
\newcommand{\om}{\omega}   
\newcommand{\si}{\sigma}   
\newcommand{\BR}{\mathbb{R} } 
\newcommand{\BZ}{\mathbb{Z} }

\newcommand{\la}{\lambda}

\newcommand{\gaa}{\gamma}

\newcommand{\RR}{\mathbb{R}}
\newcommand{\BC}{\mathbb{C}}

\newcommand{\sm}{\setminus}
\renewcommand{\phi}{\varphi}

\newcommand{\cD}{\mathcal{D}}

\newcommand{\T}{\mathcal{T}}

\newcommand{\cDA}{\cD(A)}
\newcommand{\cDL}{\cD(L)}

\begin{document}
\sloppy

\title[One-dimensional perturbations with empty spectrum] {One-dimensional
perturbations \\ of unbounded selfadjoint operators \\ with empty
spectrum}

\author{Anton D. Baranov}
\address{
Department of Mathematics and Mechanics, Saint Petersburg State
University, 28, Universitetski pr., St. Petersburg, 198504,
Russia} \email{anton.d.baranov@gmail.com}

\author{Dmitry V. Yakubovich}
\address{Departamento de Matem\'{a}ticas,
Universidad Autonoma de Madrid, Cantoblanco 28049 (Madrid) Spain
\enspace and
\newline
\phantom{r} Instituto de Ciencias Matem\'{a}ticas (CSIC - UAM -
UC3M - UCM)} \email{dmitry.yakubovich@uam.es}
\thanks{The first author has been supported by the Chebyshev Laboratory
(St. Petersburg State University) under RF Government grant 11.G34.31.0026
and by RFBR grant 11-01-00584-a.
The second author has been supported by the
Project MTM2008-06621-C02-01, DGI-FEDER, of the Ministry of
Science and Innovation, Spain.}


\begin{abstract}
We study spectral properties of
one-dimensional singular perturbations of an unbounded
selfadjoint operator and
give criteria for the possibility to remove the whole spectrum
by a perturbation of this type.
A counterpart of our results for the case
of bounded operators provides a complete description of compact selfadjoint
operators whose rank one perturbation is a Volterra operator.

\medskip

\noindent {\bf M.S.C.(2000):}  Primary: 47A55; Secondary: 47B25, 47B07.

\end{abstract}
\maketitle

{} \vskip-2cm

\section{Introduction}

We study {\it singular} rank one perturbations of an unbounded
selfadjoint operator. This paper is a continuation of
\cite{bar-yak}, where the completeness of eigenvectors of these
perturbations was considered.

Let $\mu$ be a {\it singular} measure on $\RR$ and let $\A$ be the operator of
multiplication by the independent variable $x$ in $L^2(\mu)$ (thus, $\A$
is a cyclic {\it singular} selfadjoint operator). Moreover, we
assume that $0 \notin {\rm supp}\, \mu$, and
so $\A^{-1}$ is a bounded operator in $L^2(\mu)$.

Now we define {\it singular rank one perturbations} of $\A$.
Let $a,b$ be functions such that \beqn \label{cond-a-b}
\frac{a}{x}, \  \frac{b}{x}  \in L^2(\mu), \neqn however,
possibly, $a,b \notin L^2(\mu)$. Let $\deab \in \BC$ be a
constant such that
\begin{align}
\label{1+de-a-b}
\begin{aligned}
\deab&\ne
\int_\RR x^{-1}a(x)\overline{b(x)} \, d\mu(x) \\
&\qquad\qquad\qquad \text{ in the case when $a \in L^2(\mu)$.}
\end{aligned}
\end{align}

We associate with any such data $(a,b,\deab)$ a linear operator
$\LL=\LL(\A, a,b,\deab)$, defined as follows:
\beqn
\begin{aligned}
\label{2} \mathcal{D}(\LL) &\defin \big\{
y=y_0+c\cdot \A^{-1}a: \\
& \qquad \qquad c\in \BC,\, y_0\in \mathcal{D}(\A),\, \deab c+\langle y_0,
b\rangle=0
\big\};                     \\
\LL y& \defin \A y_0, \quad y\in \mathcal{D}(\LL).
\end{aligned}
\neqn

Condition \eqref{1+de-a-b} is equivalent to the uniqueness of the
decomposition $y=y_0+c\cdot \A^{-1}a$ in the above formula for
$\mathcal{D}(\LL)$, hence the operator $\LL$ is correctly defined.
The operator $\LL=\LL(\A, a,b,\deab)$ is said to be a
{\it singular rank one perturbation} of $\A$.

Singular perturbations of selfadjoint operators
have been studied for a long time, see, for instance,
\cite{Lyants-Stor-book, alb-kur-2000, Posilicano}.

Essentially, singular rank one perturbations are
unbounded algebraic inverses to bounded rank one perturbations
of bounded selfadjoint operators.
Namely, if the triple $(a,b,\deab)$ satisfies \eqref{1+de-a-b}
and $\deab \ne 0$, then
the bounded operator
$\A^{-1}- \deab^{-1} \A^{-1}a\, (\A^{-1}b)^*$
has trivial kernel, and
$$
\LL(\A, a,b, \deab)=\big(\A^{-1}-\deab^{-1} \A^{-1}a\, (\A^{-1}b)^*\big)^{-1}.
$$
Here we denote by $\A^{-1}a\, (\A^{-1}b)^*$ the bounded rank one operator
$\A^{-1}a\, (\A^{-1}b)^* f =  (f, \A^{-1}b) \A^{-1}a$, $f\in L^2(\mu)$.
Conversely, if $\A_0$ is a bounded selfadjoint operator
with trivial kernel and $\LL_0 = \A_0 + a_0 b_0^*$ is its
rank one perturbation and $\Ker \LL_0 = 0$, then the algebraic inverse
$\LL_0^{-1}$ is a singular rank one perturbation of $\A_0^{-1}$.
We refer to \cite{bar-yak} for details and for
similar statements for rank $n$ singular perturbations.

During the last 20 years selfadjoint
rank one perturbations of selfadjoint
operators were extensively studied by Simon, del Rio, Makarov
and many other authors in relation with the problem of stability
of the point spectrum and the study of the singular continuous spectrum
(see \cite{simon-makarov} and a survey \cite{simon}). Some
recent developments can be found in \cite{treil, alb-kokoshm05}).
In what follows, we consider only perturbations of compact selfadjoint
operators (or operators with compact resolvent), but the perturbations
are no longer selfadjoint. The spectral structure
of this class becomes unexpectedly rich and complicated
as soon as we leave the classes covered by classical theories
(weak perturbations in the sense of Macaev or dissipative operators).

In our preceding paper \cite{bar-yak}, we studied
the completeness of eigenvectors of $\LL$ and $\LL^*$ as well as the possilbility
of the spectral synthesis for such perturbations.
Our main tool in \cite{bar-yak} was a functional model for rank one singular
perturbations. This model realizes singular rank one perturbations as certain
`shift' operators in a so-called model subspace of the Hardy space or
%
%
in a de Branges space of entire functions.

In this paper, we address the following problem:
\medskip
\\
{\bf Problem 1.} \textit{For which measures $\mu$ does there exist a singular
perturbation $\LL$ of $\A$ of the above type whose spectrum is
empty?}
\medskip

Clearly, if such a perturbation exists, then the resolvent of $\A$
is compact, and so the measure $\mu$ in question
should necessarily be of the form $\mu = \sum_n \mu_n \delta_{t_n}$, where
$t_n \in \RR$ and $|t_n| \to\infty$, $|n| \to \infty$. Here $\{t_n\}$
may be either one-sided sequence (enumerated by $n \in \mathbb{N}$)
or a two-sided sequence (enumerated by $n \in \mathbb{Z}$).
Thus, the problem is to describe those spectra $\{t_n\}$
for which the spectrum of the perturbation is empty.
Such spectra will be said to be {\it removable}. It is clear that the
property to be removable or nonremovable  depends only on $\{t_n\}$, but
not on the choice of the masses $\mu_n$.

The change of boundary conditions of an ordinary differential operator
leads to a singular one-dimensional perturbation, see, for instance, \cite{bar-yak}.
This phenomenon of the disappearance of the spectrum if the
boundary conditions are properly chosen is well-known, see
\cite{Naim-book, Khromov, Biy_Dzhumb}. As an example,
consider the simplest first order
selfadjoint operator $\A f(t) = - i f'(t)$ on $[0,2\pi]$ with
the boundary condition $f(2\pi) = f(0)$, whose spectrum is $\BZ$.
The operator $\LL f(t) = - i f'(t)$ with the changed boundary
condition $f(0) = 0$ satisfies $\A=\LL$ on $\cDA\cap \cDL$; moreover,
$\cDA\cap \cDL$ has codimension one both in $\cDA$ and in $\cDL$.
Therefore $\LL$ is a rank one singular perturbation of $\A$ (see \cite{bar-yak}). Since the spectrum of
$\LL$ is empty, the spectrum $\sigma(\A)=\BZ$ is removable.

In view of the relation between singular rank one perturbations
and usual rank one perturbations of bounded selfadjoint operators,
the problem is equivalent to the following:
\medskip
\\
{\bf Problem 2.} {\it Describe those compact selfadjoint operators
that have a rank one perturbation which is
a Volterra operator.}
\medskip

Recall that a compact operator is called a Volterra operator if
its spectrum equals $\{0\}$
(sometimes the assumption that the kernel is trivial is also included
in the definition; all Volterra operators appearing in the present paper
have this property).
There exist a vast range of results (mainly due to Krein, Gohberg
and Macaev) relating the Schatten class properties of the imaginary
part of a Volterra operator with the corresponding property for its real part.
See \cite[Ch. IV]{Gohb_Krein} or \cite[Ch. III]{Gohb_Kr_Volterr}
for these results and for their generalizations to more general
symmetric norm ideals; see also a more recent work \cite{MatsSod}.
Also, in \cite[Ch. IV, \S 10]{Gohb_Krein}
some partial results are given
about the spectra of Volterra operators
with finite-dimensional imaginary part. Let us also mention
a remarkable theorem, which essentially goes back to
Liv\v sic \cite{liv}
(for an explicit statement see \cite[Ch. I, Th. 8.1]{Gohb_Kr_Volterr}):
{\it  any dissipative Volterra operator, which is a
rank one perturbation of a selfadjoint operator,
is unitary equivalent to the integration operator. }
In Section \ref{examples_rem} we discuss this theorem in more detail
and deduce it from out model.
Another group of results is concerned with the existence of bases
of eigenvectors for rank one perturbations of Volterra integral operators
\cite{Khromov, malam, Romaschenko}.

In the present paper we analyze the situation where
the imaginary part of a Volterra operator
$\LL$ is (at most) a rank two operator and look for
the description of the possible spectra of the real part.
As far as we know, this particular problem was not previously
\margp{Gohb Krein Volterrovy ?} considered.

Probably, the closest results to ours were obtained by Silva and Tolosa
\cite{st1} (see also their paper \cite{st2} for some more general results)
who described entire (in the sense of Krein) operators in terms of
spectra of two of their selfadjoint extensions. Their description
is based on a theorem due to Woracek \cite{wor} characterizing de Branges spaces,
which contain zerofree functions. Though our problem deals with only
one spectrum, its solution is based on a functional model
in a de Branges space and the problem essentially reduces to the
existence of a zerofree function in it.

Finite rank perturbations of Volterra operators and their models in de Branges spaces
also have been studied in several works by Gubreev and coathors.
These papers concern Riesz bases, completeness, generation of $C_0$ semigroups
and the relation with the so-called quasi-exponentials, see
\cite{Gubr-Lat2011,Gubr-Tar2010,LukGub} and references therein.
The paper by Khromov \cite{Khromov} treats
spectral properties of finite rank perturbations of Volterra operators
from a different point of view; in particular, it contains results
stated in terms of the asymptotics of the kernel $M(x,t)$ of an integral
Volterra operator near the diagonal.

In this paper, we solve Problems 1 and 2
and obtain a necessary and sufficient condition of removability
in terms of entire functions of the
so-called Krein class. We say that an entire function
$F$ (with $F(0) \ne 0$) is in the {\it Krein class}
$\mathcal{K}_1$, if it is real on $\RR$, has only real simple zeros
$t_n$ and may be represented as
\beqn
\label{krein}
\frac{1}{F(z)} = q+ \sum_n c_n
\Big(\frac{1}{t_n-z}-\frac{1}{t_n}\Big), \qquad \sum _n
t_n^{-2}|c_n| <\infty,
\neqn
where $c_n = -1/F'(t_n)$ and $q=1/F(0)$ (see Section \ref{entire}
and Lemma \ref{class-k1} for details).

Our main result reads as follows:

\begin{theorem}
\label{annih2}
Let $t_n \in \RR$ and $|t_n| \to\infty$, $|n| \to \infty$.
The following are equivalent:

$(i)$ The spectrum $\{t_n\}$ is removable\textup;

$(ii)$ There exists a function $F \in \mathcal{K}_1$
such that the zero set of $F$ coincides with $\{t_n\}$.
\end{theorem}

An unexpected (and rather counterintuitive) consequence of Theorem \ref{annih2} is that
adding a finite number of points to the spectrum helps it to
become removable, while deleting  a finite number of points may
make it nonremovable (see Corollaries \ref{remov} and \ref{nonremov} below).

We have an immediate counterpart of Theorem  \ref{annih2}
for compact operators which have Volterra rank one perturbations.

\begin{theorem}
\label{annih3}
Let $s_n \in \RR$, $s_n \ne 0$,
and $|s_n| \to 0$, $|n| \to \infty$, and let
$\A_0$ be a compact selfadjoint operator with simple point spectrum $\{s_n\}$.
The following are equivalent:

$(i)$ There exists a rank one perturbation $\LL_0 = \A_0 + a_0 b_0^*$
such that $\LL_0$ is a Volterra operator\textup;

$(ii)$ The points $t_n = s_n^{-1}$ form the zero set of some function
$F \in \mathcal{K}_1$.
\end{theorem}

The paper is organized as follows. In Sections \ref{fmodel} and \ref{entire}
we give some preliminaries on the functional model from \cite{bar-yak}
and on de Branges' theory. The proofs of Theorems \ref{annih2}
and \ref{annih3} are given in
Sections \ref{firstc}  and \ref{sect5}. Section \ref{examples_rem}
contains some examples of removable and nonremovable spectra, while
in Section \ref{livst} we discuss a simple proof of Liv\v sic's theorem
by our methods. In Section \ref{final} we show that sometimes the Volterra
property may be achieved by sufficiently `smooth'
perturbations and compare
this result with a classical completeness
theorem due to Macaev and some results from \cite{bar-yak}.
\bigskip
\\
{\bf Acknowledgements.} The authors are deeply grateful to N.~Nikolski
and R.~Romanov for many valuable discussions.
\bigskip


\section{Functional model}
\label{fmodel}

We use the notations $\BC^\pm=\big\{z\in \BC: \pm \ima z>0\big\}$
for the upper and the lower half-planes and set $H^2=H^2(\BC^+)$.
Recall that a function $\Theta$ is said to be an {\it inner
function} in $\BC^+$ if it is a bounded analytic function
with $|\Theta| =1$ a.e. on $\RR$ in the sense of nontangential boundary values.
Each inner function $\Theta$ generates a {\it shift-coinvariant} or {\it model
subspace} $K_\Theta \defin H^2\ominus \Theta H^2$ of the Hardy space $H^2$
(we refer to, e.g., \cite{nk12} for the theory of model spaces and for their numerous applications).

The following functional model of singular rank one perturbations
was constructed in \cite{bar-yak}.
Essentially, it is similar to functional models by Kapustin
\cite{Kap} (for the rank one perturbations of unitary operators)
and Gubreev and Tarasenko \cite{Gubr-Tar2010} (for the compact resolvent case);
for a more general model see \cite{Ryzhov}.
In what follows we assume that $b$ is a {\it cyclic}
vector for the resolvent of $\A$, so $b\ne 0$ $\mu$-a.e.

\begin{theorem}[\cite{bar-yak}, Theorem 0.6]
\label{rank-one-model}
Let $\LL=\LL(\A,a,b, \deab)$ be a singular rank one perturbation of $\A$,
and let $b$ be a {\it cyclic} vector for $\A^{-1}$.
Then there exist an inner function $\Theta$, such that
$\Theta$ is analytic in a neighborhood
of $0$,  $1+\Theta \notin H^2$, $\Theta(0) \ne -1$, and a function $\phi$ satisfying
\beqn
\label{main0}
\phi\notin H^2, \qquad
\frac{\phi(z) -\phi(i)}{z-i} \in K_\Theta,
\neqn
such that $\LL$ is unitary equivalent to the operator $\T=\T(\Theta,\phi)$
which acts on the model space $K_\Theta \defin H^2\ominus \Theta H^2$
by the formulas
$$
\cD(\T) \defin \{f=f(z)\in K_\Theta: \text{there exists}\ c=c(f)\in \BC:
zf-c\phi\in K_\Theta\},
$$
$$
\T f \defin zf - c \varphi, \qquad f\in \cD(\T).
$$

Conversely, any inner function $\Theta$ which is analytic in a neighborhood
of $0$ and satisfies $1+\Theta \notin H^2$, $\Theta(0) \ne -1$,
and any function $\phi$ satisfying \eqref{main0}
correspond to some singular rank one perturbation
$\LL= \LL(\A, a, b, \deab)$ of the operator $\A$ of multiplication by the independent
variable in $L^2(\mu)$, where $\mu$ is some singular measure on $\RR$ and
$x^{-1}  a(x)$, $x^{-1}  b(x) \in L^2(\mu)$.
\end{theorem}

The functions $\Theta$ and $\phi$ appearing in the model for
$\LL(\A,a,b, \deab) $ are given by the following formulas. Put
\beqn
\label{singul2}
\begin{aligned}
\beta(z) & = \deab + z b^*(\A-z)^{-1} \A^{-1}a \\
        & = \deab + \int \bigg(\frac{1}{x-z}-\frac{1}{x}\bigg)\, a(x)\overline{b(x)} \, d\mu (x),
\end{aligned}
\neqn
\beqn
\label{singul1}
\rho(z)
 = \de + zb^*(\A-z)^{-1} \A^{-1}b  =
\de + \int \bigg(\frac{1}{x-z}-\frac{1}{x}\bigg) |b(x)|^2
\, d\mu (x),
\neqn
where $\delta$ is an arbitrary real constant. Then $\Theta$ and $\phi$
are defined as
\beqn
\label{the}
\Theta(z)=  \frac{i-\rho(z)}{i+\rho(z)}, \qquad \quad
\phi(z)=\frac {\beta(z)} 2\, \big(1+\Theta(z)\big).
\neqn

The above functional model uses essentially the properties
of the so-called Clark measures introduced in \cite{cl}.
Recall that the Clark measure $\sigma_\zeta$, $|\zeta| =1$, is the
measure from the Herglotz representation
$$
i\,\frac{\zeta+\Theta(z)}{\zeta-\Theta(z)}=  p_\zeta z + q_\zeta
+\frac{1}{\pi}
\int_\mathbb{R} \Big(\frac{1}{t-z} -\frac{t}{t^2+1} \Big)
d \sigma_\zeta(t),  \qquad z\in\mathbb{C^+},
$$
where $p_\zeta \ge 0$, $q_\zeta\in \mathbb{R}$
and $\int_\RR (1+t^2)^{-1} d\sigma_\zeta(t)<\infty$.
Note that if $\Theta$ is meromorphic,
then any Clark measure $\sigma_\zeta$ is discrete.

It follows from the results of Ahern and Clark \cite{ac70} that
\beqn
\label{mass}
\zeta-\Theta \in H^2 \ \Longleftrightarrow \ p_\zeta>0.
\neqn

Note that in our model
\beqn
\label{the1}
i\, \frac{1-\Theta(z)}{1+\Theta(z)} =
\de + \int \bigg(\frac{1}{x-z}-\frac{1}{x}\bigg) |b(x)|^2
\, d\mu (x).
\neqn
Thus, the measure $\pi |b|^2 \mu$ is the Clark measure $\sigma_{-1}$
for $\Theta$.

Let us mention the following result of \cite{bar-yak}.

\begin{proposition}[\cite{bar-yak}, Proposition 2.2]
\label{param}
Let $a,b$ be functions that satisfy \eqref{cond-a-b} and let $\deab\in \RR$.
Let $\Theta$ and $\phi$ be defined by \eqref{the}.
Then we have:

1. $1+\Theta\notin H^2$, $\Theta(0)\ne -1$, and $\frac{\phi}{z+i} \in H^2$\textup;

2. If $a\notin L^2(\mu)$, then $\phi \notin H^2$\textup;

3. If $a\in L^2(\mu)$, then $\phi \in H^2$ if and only if
$\deab=\sum_n a_n\overline b_n t_n^{-1} \mu_n$.

\end{proposition}

Since we are interested in the case when $\mu$ is a discrete measure:
$\mu = \sum_n \mu_n \delta_{t_n}$,
where $|t_n| \to\infty$, $|n| \to \infty$,
the function $\Theta$ is meromorphic in the whole complex plane and
analytic on $\RR$; so is $\phi$ and any element of $K_\Theta$.
This situation reduces to the study of de Branges spaces of entire
functions (see Section \ref{entire} below).

From now on we assume that $\Theta$ and $\phi$ are meromorphic.
By the well-known properties of the model spaces $K_\Theta$,
a function $f\in H^2(\BC^+)$ is in $K_\Theta$ if and only if the
function $\tilde f(z) = \Theta(z) \overline{f(\overline z)}$
also is in $H^2(\BC^+)$. Analogously, we put
\[
\tilde \phi(z) = \Theta(z) \overline{\phi(\overline z)}.
\]

Denote by $Z_\phi$ the set of zeros of $\phi$ in
$\clos\BC^+$ and
put  $\overline{Z}_{\tilde \phi} = \{z\in\clos\BC^-: \overline z \in Z_{\tilde \phi}\}$.
It follows from \cite[Lemma 2.1]{bar-yak} that the functions
$$
h_\lambda(z) = \frac{\phi(z)}{z-\lambda}, \qquad \lambda\in Z_\phi\cup
\overline{Z}_{\tilde\phi}
$$
belong to $K_\Theta$, and, moreover,
all eigenfunctions of the model operator $\T$
are of the form $h_\lambda$, $\lambda\in Z_\phi\cup
\overline{Z}_{\tilde \phi}$.

\begin{lemma}[\cite{bar-yak}, Lemma 2.4]
\label{eigs}
Let meromorphic $\Theta$ and $\phi$ correspond to a
singular rank one perturbation of a cyclic selfadjoint operator
$\A$ with the compact resolvent. Then the following holds:

1. Operators $\LL$ and $\T$ have compact resolvents\textup;

2. $\si(\T)=\si_p(\T)=Z_\phi \cup \overline Z_{\tilde \phi}$\textup;

3. The eigenspace of $\T$ corresponding to an eigenvalue
$\la\in Z_\phi \cup \overline Z_{\tilde \phi}$, is spanned by
$h_\la$.
\end{lemma}
\medskip


\section{Preliminaries on entire functions}
\label{entire}

An entire function $E$ is said to be in the {\it Hermite--Biehler class}
(which we denote by $HB$) if
$$
|E(z)|>|E(\overline z)|, \qquad z\in \BC^+.
$$
We also always assume that $E \ne 0$ on $\RR$.
For a detailed study of the Hermite--Biehler class see
\cite[Chapter VII]{levin}.
Put $E^*(z) = \overline{E(\overline z)}$. If $E\in HB$, then
$\Theta = E^*/E$
is an inner function which is meromorphic in the whole plane $\BC$;
moreover, any meromorphic inner function
can be obtained in this
way for some $E\in HB$ (see, e.g., \cite[Lemma 2.1]{hm1}).

Given $E\in HB$, we can always write it as $E= A-iB$, where
$$
A=\frac{E+E^*}{2},\qquad B=\frac{E^*-E}{2i}.
$$
Then $A$, $B$ are real on the real axis and have simple real zeros.
Moreover, if $\Theta =E^*/E$, then $2A = (1+\Theta)E$.

Any function $E\in HB$ generates the de Branges space
${\mathcal H} (E)$ which consists of all entire functions $f$
such that $f/E$ and $f^*/E$ belong
to the Hardy space $H^2$, and $\|f\|_E = \|f/E\|_{L^2(\mathbb{R})}$
(for the de Branges theory see \cite{br}).
It is easy to see that the mapping $ f\mapsto f/E $
is a unitary operator from ${\mathcal H}(E)$ onto  $K_\Theta$
with $\Theta=E^*/E$ (see, e.g., \cite[Theorem 2.10]{hm1}).

An entire function $F$ is said to be of {\it Cartwright class}
if it is of finite exponential type and
$$
\int_{\RR} \frac{\log^+|F(x)|}{1+x^2}dx <\infty
$$
(recall that $\log^+ t = \max (\log t, 0)$, $t>0$).
For the theory of the Cartwright class we refer to \cite{hj, ko1}.
It is well-known that the zeros $z_n$ of a Cartwright class function
$F$ have a certain symmetry: in particular,
\beqn
\label{cartw}
F(z) = K z^m e^{icz}\, v.p. \prod_n \Big(1- \frac{z}{z_n}\Big)
\defin
K z^m e^{icz}
\lim_{R\to \infty} \prod_{|z_n|\le R} \Big(1- \frac{z}{z_n}\Big),
\neqn
where the infinite product converges in the
`principal value' sense, $c\in \RR$ and $K\in \BC$ are some constants,
$m\in\mathbb{Z}_+$.

It follows from this representation that a Cartwright class
function is determined uniquely by its zeros, up to a factor $K e^{i \gamma z}$, where
$K\in \BC$ and $\gamma\in \BR$ are constants.

A function $f$ analytic in $\BC^+$ is said to be of {\it bounded type} if
$f=g/h$ for some functions $g$, $h\in H^\infty(\BC^+)$.
If, moreover, $h$ can
be taken to be outer, we say that $f$
is in \textit{the Smirnov class in $\BC^+$}.
It is well known that if $f$ is analytic in $\BC^+$
and $\ima f>0$ (such functions are said to be the {\it Herglotz functions}),
then $f$ is in the Smirnov class  \cite[Part 2, Chapter 1, Section 5]{hj}.
In particular, if $t_n\in \mathbb{R}$, $u_n>0$ and $\sum_n u_n<\infty$,
then  the function $\sum_n \frac{u_n}{t_n-z}$ is in the Smirnov class
in $\BC^+$. Consequently,
$\sum_n \frac{v_n}{t_n-z}$ is in the Smirnov class in $\BC^+$
for any $\{v_n\}\in \ell^1$.

The following theorem due to M.G. Krein (see, e.g.,
\cite[Part II, Chapter 1]{hj}) will be useful:
{\it If an entire function $F$ is of bounded type both in $\BC^+$
and in $\BC^-$, then $F$ is of Cartwright class.
If, moreover, $F$ is in the Smirnov class
both in $\BC^+$ and in $\BC^-$, then $F$
is a Cartwright class function of zero exponential type.}

We also consider the class of entire functions introduced by M.G.
Krein \cite{krein47} (see also \cite[Chapter 6]{levin}). Assume
that $F$ is an entire function, which is real on $\RR$,
with simple real zeros $t_n \ne 0$ such that,
for some integer $p\ge 0$, we have
$$
\sum_n \frac{1}{|t_n|^{p+1} |F'(t_n)|}<\infty
$$
and \beqn \label{krein0} \frac{1}{F(z)} = R(z) + \sum_n
\frac{1}{F'(t_n)} \cdot \bigg(\frac{1}{z-t_n} +\frac{1}{t_n}+
\frac{z}{t_n^2}+\cdots + \frac{z^{p-1}}{t_n^p} \bigg), \neqn where
$R$ is some polynomial. The class of such functions $F$ we will denote by
$\mathcal{K}_p$. If $F \in \mathcal{K}_p$ for some $p$, then $F$ is of
Cartwright class \cite[Chapter 6]{levin}.
\bigskip


\section{First criterion of removability}
\label{firstc}

Recall that the spectrum $\{t_n\}$ with
$0\notin \{t_n\}$ is said to be {\it
removable} if there exists a singular perturbation $\LL = \LL(\A, a, b, \deab)$
of $\A$, whose spectrum is empty.
Here $\A$ is the operator of multiplication by $x$
in $L^2(\mu)$, $\mu = \sum_n \mu_n \delta_{t_n}$.
In this case, given $a, b \in x L^2(\mu)$ we
will write $a_n$ and $b_n$ in place
of $a(t_n)$ and $b(t_n)$.

It is obvious that if the spectrum of $\LL=\LL(\A, a, b, \deab)$
is empty, then $b$ must be a cyclic vector for $\A^{-1}$. In fact, if
$b_n =0$ then the vector $e_n$ defined by $e_n(t_k) = \delta_{nk}$
will be an eigenvector of $\LL$ corresponding to the eigenvalue $t_n$.
Indeed, $e_n$ belongs to $\mathcal{D}(\LL)$ since $(e_n, b) = 0$ and we may take $c=0$.

Since $b$ is cyclic for $\A^{-1}$, we may apply the functional model
from Section \ref{fmodel}.
Then, in view of Theorem \ref{rank-one-model} and Lemma \ref{eigs},
we have an immediate criterion of removability.

\begin{proposition}
\label{annih-1}
The spectrum $\{t_n\}$ is removable
if and only if there exist
a meromorphic inner function $\Theta$ with $\{t: \Theta(t)=-1\} =
\{t_n\}$, $\Theta(0) \ne -1$, $1+\Theta \notin H^2$, and a function
$\phi$ which satisfies \eqref{main0} such that both $\phi$ and $\tilde \phi$
have no zeros in $\BC^+\cup\RR$.
\end{proposition}

\begin{proof}[Proof]
Indeed, if such pair $(\Theta, \phi)$ exists, then, by the converse
statement in Theorem \ref{rank-one-model} there is
a singular measure $\mu$, functions $a$ and $b$ and a constant $\deab$
such that $\LL(\A, a,b, \deab)$ is unitarily equivalent to the model operator
$\T(\Theta, \phi)$. Moreover, in this case $\Theta$ and $\phi$ are related
to the data $(a,b, \deab)$ by formulas (\ref{singul2})--(\ref{the}).
Thus, $|b|^2\mu $ is the Clark  measure $\sigma_{-1}$ for $\Theta$
whence $\{t: \Theta(t)=-1\} = \{t_n\}$.
By  Lemma \ref{eigs} the spectrum of $\T(\Theta, \phi)$ is empty
if and only if $\phi$ and $\tilde \phi$ do not vanish in $\BC^+\cup\RR$.
\end{proof}

The following statement gives a more palpable description of removable spectra.
In particular, we will see that the function $\phi$ may be chosen of the form
$1/E$ for a function $E$ in the Hermite--Biehler class.

\begin{theorem}
\label{annih}
The spectrum $\{t_n\}$ is removable if and only if
the
following two conditions hold:

\textup(1\textup) The set $\{t_n\}$ is the zero set of an entire function in
the Cartwright class, and so the \emph{generating function} of the
set $\{t_n\}$,
\beqn
\label{cart}
A(z) = v.p. \prod \Big(1-
\frac{z}{t_n}\Big)  = \lim_{R\to \infty} \prod_{|t_n|\le R}
\Big(1- \frac{z}{t_n}\Big),
\neqn
is well-defined and belongs to the Cartwright class;

\textup(2\textup) Moreover,
there exists an entire function $E$ of the Hermite--Biehler
class such that $ E+E^* = 2A$, $A\notin \mathcal{H}(E)$
and $\frac{1}{(z+i)E} \in H^2$.

%
%
%
If the spectrum is removable \textup(so that \textup(1\textup) 
and \textup(2\textup) hold\textup), then
the pair $(\Theta, \phi)$, corresponding to a perturbation
of $\A$ with empty spectrum
and a function $E$ in \textup(2\textup) may be chosen so that $\Theta = E^*/E$ and $\phi = 1/E$.
\end{theorem}

%
%

%
%
%
%
%

\begin{proof} 
{\it Necessity of 1 and 2.} If the spectrum $\{t_n\}$ is removable, then
there is a pair $(\Theta, \phi)$ satisfying all conditions in Proposition
\ref{annih-1}. Since $\phi$ is a function of bounded type (and even of Smirnov class)
which does not vanish in $\BC$ and is analytic in a neighborhood of $\RR$,
its inner-outer factorization (see, e.g., \cite{ko0})
is of the form
$$
\phi(z) = e^{ic_1 z} {\mathcal O}(z),
$$
where $c_1 \ge 0$ and ${\mathcal O}$ is an outer function with $|{\mathcal O}| = |\phi|$ on $\RR$.
Since $\tilde \phi$ is also a function of Smirnov class, we have for $t\in\RR$,
$$
\tilde\phi (t) = \Theta(t) \overline{{\mathcal O}(t)}e^{-ic_1t} = {\mathcal O}(t) e^{ic_2t} = \phi(t)e^{i(c_2-c_1)t}
$$
for some $c_2 \ge 0$. We conclude that for $z\in \RR$ and, hence, for any $z\in \BC$,
\beqn
\label{dfg}
e^{2icz} = \frac{\phi(z)}{\tilde\phi(z)},
\neqn
where $2c=c_1 -c_2$.

The function $1/\phi$ is a meromorphic function, which has no poles
in $\BC^+$ and on $\RR$. Also, by (\ref{dfg}),
$$
\frac{1}{\overline{\phi(\overline z)}} =
\frac{\Theta(z)}{\phi(z)}\,e^{2icz}, \qquad z\in \BC^+,
$$
and so $1/\phi$ has no poles in $\BC^-$. We conclude that
$E= e^{icz}/\phi$ is an entire function. The function $E$ is in $HB$,
because $E^*/E =\Theta$. Also, $E$ is of bounded type both in the
upper and the lower half-planes, and so is of Cartwright class by
Krein's theorem.



Now put $\phi_1 = e^{-icz}\phi = 1/E$.
We assert that we can replace $\phi$ with $\phi_1$, that is, that $\phi_1$
also satisfies the conditions of Proposition \ref{annih-1}.
Indeed, we know that
$$
\frac{\phi}{z+i} = \frac{e^{icz}}{(z+i)E} \in H^2
$$
(see \eqref{main0}).
We assert that $\frac{\phi_1}{z+i} = e^{-icz}\, \frac{\phi}{z+i}$ is also in $H^2$.
If $c\le 0$, then it is obvious. On the other hand,
since $E\in HB$ is of order at most one and does not vanish on $\RR$,
it admits the following factorization (see, e.g., \cite[Chapter VII]{levin})
$$
E(z)= K e^{-iaz +bz} \prod_{n}\bigg(1-\frac
        {z}{z_n}\bigg) e^{h_n z},
$$
where $K\in \BC$, $a\ge 0$, $b\in \RR$,
$\{z_n\}$ is a finite or infinite sequence of points in $\BC^-$,
satisfying the Blaschke condition, and  $h_n =\rea \frac{1}{z_n}\ge 0$.
It follows that $|E(iy)| \to\infty$ when $y\to\infty$. Hence, when $c>0$,
$\frac{\phi(iy)}{i(y+1)} = o(e^{-cy})$, and, thus, the function
$\frac{\phi}{z+i}$ is divisible by $e^{icz}$ in $H^2$.
Since $\phi\notin H^2$, it follows that $\phi\notin L^2(\RR)$, so that
$\phi_1\notin H^2$. Next, since $\tilde \phi_1=\phi_1$,
it follows that $\Theta \,\frac {\bar \phi_1(t) -\overline{\phi_1(i)}}{t+i}\,\big|_\RR$ is in $H^2$, which
implies that $\frac {\phi_1(z) -\overline{\phi_1(i)}}{z-i}\in K_\Theta$. Hence
\eqref{main0} holds.

We get that
$E$ is both in the Hermite--Biehler and in the Cartwright class
and satisfies
$\frac{1}{(z+i)E} = \frac{\phi_1}{z+i} \in H^2$.
Then $A = \frac{E+E^*}{2}$ is a Cartwright class
function with zero set $\{t:\, \Theta(t)=-1\} = \{t_n\}$.
Since $1+\Theta \notin H^2$, we conclude that
$A = \frac{(1+\Theta)E}{2} \notin \mathcal{H}(E)$.

Notice that this argument also proves the last statement of the Theorem.
\bigskip
\\
{\it Sufficiency of 1 and 2.} If $E$ and $A$ are given,
we may put $\Theta = E^*/E$ and $\phi = 1/E$. Then $\Theta$ and $\phi$
satisfy all conditions of Proposition \ref{annih-1}, except, may be, one:
it may happen that $\phi = 1/E \in H^2$.

However, the function $E$ such that
$A=\frac{E+E^*}{2}$ is not unique.  Namely, the function
$E_*=A-iB_*$ is in $HB$ if and only if $B_*/A$ is a Herglotz function
in $\BC^+$, which holds if and only if
there exist $\nu_n \ge 0$,
$\sum_n t_n^{-2} \nu_n<\infty$, $p\ge 0$ and $r\in\RR$ such that
\beqn
\label{g-par}
\frac{B_*(z)}{A(z)} = pz +  r + \sum_n \nu_n
\Big(\frac{1}{t_n-z} -\frac{1}{t_n}\Big).
\neqn
In particular, there exist $\nu_n^0$ and $r_0$ such that
for our initial $E=A-iB$ we have
$$
\frac{B(z)}{A(z)} = r_0 + \sum_n \nu_n^0 \Big(\frac{1}{t_n-z}
-\frac{1}{t_n}\Big).
$$
Since $1+\Theta \notin H^2$ and
$\frac{B}{A} = i\frac{1-\Theta}{1+\Theta}$, the corresponding
summand $p_0 z$ is absent by \eqref{mass}.
Now consider the functions $B_*$ such that
$$
\frac{B_*(z)}{A(z)} = r_0 + \sum_n \nu_n \Big(\frac{1}{t_n-z}
-\frac{1}{t_n}\Big),
$$
where $\nu_n$ are free parameters; then
$$
E_*(z) = A(z) \bigg(1 - ir_0 - i \sum_n \nu_n \Big( \frac{1}{t_n -z} -
\frac{1}{t_n}\Big) \bigg).
$$
Clearly, $E_*(t_n) = i \nu_n A'(t_n)$, and choosing $\{\nu_n\}$
rapidly decreasing we can achieve that $\frac{1}{E_*} \notin
L^2(\RR)$. On the other hand, for the choice of $\nu_n=\nu_n^0$ we
have $E_*=E$, and so $\frac{1}{(x+i)E} \in L^2(\RR)$. Since $E_*$
is a `continuous function of $\nu_n$', it
is not difficult to show that 
there exist data $\{\nu_n\}$ such that $\frac{1}{(x+i)E_*} \in L^2(\RR)$,
whereas $\frac{1}{E_*} \notin L^2(\RR)$.
\medskip

For the convenience of the reader who might be not satisfied with the above 
'continuity' argument, we give a rigorous proof of
the existence of such sequence $\{\nu_n\}$.
It may be assumed that the sequence $\{t_n\}$ has infinitely many positive terms.
We will choose a rapidly increasing
subsequence $\{t_{n_k}\}_{k=1}^\infty$ of $\{t_n\}$ such that
$t_{n_k}\to +\infty$. We will set
\[
E_*(z) = A(z) \eta_*(z),
\quad
\text{ where }
\eta_*(z) =
1-ir_0 - i \sum_n \nu_{n,*} \Big(\frac{1}{t_n-z}
-\frac{1}{t_n}\Big),
\]
with
\[
\nu_{n,*}=
\begin{cases}
\nu^0_{n}, \quad n\ne n_k, \\
\nu'_k, \quad n = n_k ,
\end{cases}
\]
where $0<\nu'_k \le \nu^0_{n_k}$.
We will also define auxiliary points $\tau_k$ such that
$t_{n_{k-1}} \le \tau_{k-1} \le t_{n_k}$ for $k\ge 2$.
The sequences $\{n_k\}$, $\{\tau_k\}$ and the weights $\{\nu'_k\}$
will be defined by induction. To do that, we first introduce some more notation.
For $k\ge 0$, let
\[
\nu_{n,k}=
\begin{cases}
\nu^0_{n}, \quad n\ne n_\ell \text{ or } n = n_\ell, \enspace \ell>k, \\
\nu'_\ell, \quad n = n_\ell, \enspace 1\le \ell\le k,
\end{cases}
\]
denote the weight, changed only in the points $t_{n_1}, \dots, t_{n_k}$.
Put
\[
E_k(z) = A(z) \eta_k(z), \quad \text{ where } \eta_k(z) = 1-ir_0 -
i \sum_n \nu_{n,k} \Big(\frac{1}{t_n-z} -\frac{1}{t_n}\Big)
\]
(so that $E_0=E$).
It is easy to see that for any such choice, $\frac 1 {(x+i)E_k}\in L^2(\RR)$ for all $k$.

The inductive definition is as follows.
On the first step, choose any $t_{n_1}>0$ and any
$\tau_1>\max(4, 2t_{n_1})$.
Now suppose that the numbers
$t_{n_\ell}$, $\nu'_\ell$ and $\tau_\ell$ ($\ell= 1, \dots, k-1$) have been already chosen.
On  the $k$th step,  $n_k$, $\nu'_k$ and $\tau_k$ will be defined.
We will use the notation $J_\ell=[-\tau_\ell, \tau_\ell]$.

Choose $t_{n_k}>2\tau_{k-1}$ so that
$\nu_{n_k}t_{n_k}^{-2}\le 2^{-k-1} \tau _{k-1}^{-1}$.
It is possible because $\sum_n t_n^{-2}\nu_n<\infty$.

If $\big\|\frac 1 {(x+i)E_{k-1}}\big\|_{L^2(\RR \sm J_{k-1} )} \ge  2\tau_{k-1}^{-1}$, 
then we put $\nu'_k=\nu_{n_k}$ (so that $E_k=E_{k-1}$). Otherwise,
we choose $\nu'_k\in (0, \nu_{n_k})$ so that
$\big\|\frac 1 {(x+i)E_k}\big\|_{L^2(\RR \sm J_{k-1} )} = 2\tau_{k-1}^{-1} $.
It is possible because this norm is continuous as a function of $\nu'_k$
and tends to infinity as $\nu'_k\to 0^+$. Next, in both cases choose 
$\tau_k>t_{n_k}$ such that 
$\big\|\frac 1 {(x+i)E_k}\big\|_{L^2(J_k \sm J_{k-1})} = \tau_{k-1}^{-1}$.
Notice that $\tau_k > 2\tau_{k-1}$, which gives that $\tau_k > 2^{k+1}$.


\medskip
We claim that the following properties hold.
\smallskip

(i) $|\eta_\ell(x)- \eta_{\ell-1}(x) |\le 2^{-\ell}$ on $J_k$ for $\ell > k$;

(ii) $\big\|\frac 1 {(x+i)E_k}\big\|_{L^2(\RR)}\le C$ 
for some constant $C$, independent of $k$;

(iii) The sequence of functions $\frac 1 {E_\ell}$ converges uniformly
to $\frac 1 {E_*}$ on $J_k$ for any $k$.
\smallskip
\\
These properties imply our statement. Indeed, (i) gives that for $\ell > m \ge k$,
\beqn
\label{estim}
\bigg| 1- \frac {\eta_m(x)}{\eta_\ell(x)} \bigg|\le
| \eta_\ell(x)-\eta_m(x) |
\le
\sum_{j=m}^{\ell-1} 2^{-j-1}  \le 2^{-m}
\quad \text{ for $x\in J_k$.}
\neqn
Next, (ii) and (iii) imply that the functions
$\frac 1 {(x+i)E_k}$ converge weakly to $\frac 1 {(x+i)E_*}$ in $L^2(\RR)$. In particular,
$\frac 1 {(x+i)E_*}$ is in $L^2(\RR)$. Fix some $k$.
For any $\ell>k$,
$|\eta_k/\eta_\ell|\ge \frac 12$ on $J_k$,
and therefore
\[
\bigg\|\frac 1 {E_\ell}\bigg\|_{L^2(J_k \sm J_{k-1})}
=
\bigg\|\frac {\eta_k} {\eta_\ell}\cdot \frac 1 {E_k}\bigg\|_{L^2(J_k \sm J_{k-1})}
\ge
\frac {\tau_{k-1} } 2 \, \bigg\|\frac 1 {(x+i) E_k}\bigg\|_{L^2(J_k\sm J_{k-1})}
\ge \frac 12,
\]
which by (iii) implies that
$\|\frac 1 {E_*}\|_{L^2(J_k \sm J_{k-1})} \ge\frac 1 2$
for any $k$. Therefore $\frac 1 {E_*}\notin L^2(\RR)$.

So it remains to check (i)--(iii).
\medskip

\textit{Proof of} (i): 
Let $\ell>k$, and let $x\in J_k \subset J_{\ell-1}$. Then
\[
|\eta_{\ell}(x)-\eta_{\ell-1}(x)|
=
(\nu_{n_\ell}-\nu'_\ell)\, \frac {|x|}{|t_{n_\ell}-x| t_{n_\ell}}
\le
\frac
{\nu_{n_\ell}\cdot2\tau_{\ell-1}}{t_{n_\ell}^2}
\le 2^{-\ell}.
\]
In particular, \eqref{estim} holds.
\medskip

\textit{Proof of} (ii):
For any index $k$ such that $E_k\ne E_{k-1}$, one has
\begin{align*}
\bigg\|\frac 1 {(x+i) E_k}\bigg\|^2_{L^2(\RR)}
& =
\bigg\| \frac {\eta_{k-1}} {\eta_k}
\cdot \frac 1 {(x+i) E_{k-1}} \bigg\|^2_{L^2(J_{k-1})}
+ \bigg\| \frac 1 {(x+i) E_k} \bigg\|^2_{L^2(\RR\sm J_{k-1})} \\
& \le
(1+2^{-k})^2\,
\bigg\|
\frac 1 {(x+i) E_{k-1}}
\bigg\|^2_{L^2(\RR)}
+
4 \tau_{k-1}^{-2}.
\end{align*}
Since $4 \tau_{k-1}^{-2} < 2^{-2k}$, one gets that
\[
1+\bigg\|
\frac 1 {(x+i) E_k}
\bigg\|^2_{L^2(\RR)}
\le
(1+2^{-k})^2
\bigg(1+\bigg\|
\frac 1 {(x+i) E_{k-1}}
\bigg\|^2_{L^2(\RR)}
\bigg).
\]
This inequality also holds if
$E_k=E_{k-1}$. Since $\prod_{k\ge 1} (1+2^{-k})^2$ converges
and $\frac 1 {(x+i)E_0}$ is in $L^2(\RR)$, property (ii) follows.
\medskip

\textit{Proof of} (iii):
It follows from \eqref{estim} that there are constants $C_k$ such that
$\big\|\frac 1 {E_\ell}\big\|_{L^2(J_k)}\le C_k$
for all $\ell$. Now it is easy to get from
the formulas $E_\ell= A  \eta_\ell$ and \eqref{estim}
that for any fixed interval $J_k$, $\big\{\frac 1 {E_\ell} \big\}$
is a Cauchy sequence in $C(J_k)$.
Since $\frac 1 {E_\ell}$ tend pointwise to $\frac 1 {E_*}$
on $\BR$, (iii) follows.
\end{proof}

\begin{example}
1. Let $t_n= n +\delta$, $n\in \mathbb{Z}$, $\delta \in (0,1)$.
This spectrum may be
annihilated by a one-dimensional perturbation, since we can take
$E(z) = ie^{-\pi i (z-\delta)}$  or
$E(z)= \sin \pi (z-\delta) +2i\cos \pi (z-\delta) $.
\smallskip

2. The spectrum  $\{t_n\} = \mathbb{N}$ is not removable, because
$\mathbb{N}$ is not a zero set of a Cartwright class function.
\end{example}

\begin{remark}
If $\frac{1}{(z+i)E} \in H^2$ for a Hermite--Biehler function $E$, then,
by Theorem~\ref{annih},
the spectrum $\{t_n\}$ (the zero set of $A = (E+E^*)/2$)
is removable. A number of conditions in
terms of the zeros of $E$ ensuring this inclusion
(which is equivalent to the fact that $1$ is a function
associated to $\mathcal H(E)$ in the terminology of \cite{br})
have been obtained in \cite{kw, bar-arkiv}, while a criterion in terms of zeros of $A$
and $B$ was given in \cite{wor}.
A slightly stronger property
$1\in \mathcal H(E)$ is closely related to the existence of positive minimal majorants
for $\mathcal H(E)$ \cite{hm1}.
\end{remark}
\medskip


\section{Conditions in terms of the generating function $F$. \\ 
Proof of Theorem~\ref{annih2}}
\label{sect5}

It is possible to give a complete characterization
of removable spectra solely in terms of the generating function $F$.
We will need the following simple lemma about the Krein class ${\mathcal K}_1$.

\begin{lemma}
\label{class-k1}
Let $\F(z) = v.p. \prod_n \big(1-\frac{z}{t_n}\big)$
be a Cartwright class function with simple
real zeros $t_n$. Then $\F \in \mathcal{K}_1$ if and only if
\beqn
\label{krein2} \sum_n \frac{1}{t_n^2 |\F'(t_n)|}<+\infty.
\neqn
Moreover, in this case in \eqref{krein0}, $R(z) \equiv R\equiv const$.
\end{lemma}

\begin{proof}[Proof]
We need to show only that (\ref{krein2}) implies representation
(\ref{krein0}) with $R\equiv const$. Put
$$
R(z) = \frac{1}{\F(z)} - \sum_n \frac{1}{\F'(t_n)}
\Big(\frac{1}{z-t_n} +\frac{1}{t_n}\Big).
$$
Obviously, $R$ is an entire function. Moreover, since $\F$ is in
the Cartwright class and $|\F(iy)| \to\infty$ as $y\to\infty$, we
conclude that $1/\F$ is in the Smirnov class in the upper and in
the lower half-planes. The function
$$
\sum_n \frac{1}{\F'(t_n)} \bigg(\frac{1}{z-t_n}
+\frac{1}{t_n}\bigg) =
z^2 \sum_n \frac{1}{t_n^2 \F'(t_n)} \cdot
\frac{1}{z - t_n} - z \sum_n \frac{1}{t_n^2 \F'(t_n)}
$$
is also in the Smirnov class (see Section 2.4). Thus $R$ is of
zero exponential type by Krein's theorem. Finally, note that
$|R(iy)| = o(|y|)$, $|y|\to\infty$, and so $R$ is a constant.
\end{proof}

\begin{proof}[Proof of Theorem \ref{annih2}]
Assume that the spectrum
$\{t_n\}$
is removable. Then there exist $E = A-iB$ and
$\Theta = E^*/E$ as in Theorem \ref{annih}, so that
$1+\Theta \notin H^2$.
By \eqref{the}, $\phi=1/E$ is of the form
$$
\phi(z) =\frac{1+\Theta(z)}{2}\bigg( \deab + \sum_n c_n
\Big(\frac{1}{t_n-z}-\frac{1}{t_n}\Big) \bigg),
$$
where $c_n =a_n\overline b_n\mu_n$, and so $\sum_n
t_n^{-2}|c_n|<\infty$. On the other hand,
$$
\phi = \frac{1}{E}= \frac{1+\Theta}{2A},
$$
and we conclude that $1/A$ has the representation of the form
(\ref{krein}).

Conversely, if $A \in \mathcal{K}_1$, then
$A$ is of Cartwright class and
$$
\frac{1}{A(z)}= q + \sum_n
c_n\Big(\frac{1}{t_n-z}-\frac{1}{t_n}\Big), \qquad \sum_n
\frac{|c_n|}{t_n^2}<\infty.
$$
Now for any masses $\mu_n>0$ we may choose $a_n$ and
$b_n$ so that $c_n = a_n\overline b_n\mu_n$ and
$$
\sum_n |a_n|^2t_n^{-2}\mu_n<\infty,  \qquad \sum_n
|b_n|^2t_n^{-2}\mu_n<\infty.
$$
Indeed, note that $c_n = - 1/A'(t_n )\ne 0$,
and take $a_n = |c_n|^{1/2}\mu_n^{-1/2}$ and
$b_n = \overline{c_n} |c_n|^{-1/2}\mu_n^{-1/2}$.
Define $\Theta$ by formula (\ref{the1}) (with an arbitrary real constant
$\de$). By construction, $1+\Theta \notin H^2$ (see the equivalence (\ref{mass})).
Then put
$$
E = \frac{2A}{1+\Theta}.
$$
Clearly, $E$ is an entire function (the zeros sets
of $1+\Theta$ and of $A$ coincide) and
\beqn
\label{shar3}
\frac{E^*(z)}{E(z)} = \frac{1+\Theta(z)}{1+\overline{\Theta(\overline{z})}} =
\Theta(z)
\neqn
since $\overline{\Theta(\overline{z})} = \big(\Theta(z)\big)^{-1}$.
Thus, $E$ is a Hermite--Biehler function and $A\notin \mathcal{H}(E)$
since $1+\Theta \notin H^2$.
Finally, for the function $\phi = 1/E$ we have
$$
\phi(z) = \frac{1+\Theta(z)}{2A(z)} =
\frac{1+\Theta(z)}{2}\bigg(q + \sum_n
c_n\Big(\frac{1}{t_n-z}-\frac{1}{t_n}\Big) \bigg).
$$
We see that $\phi$ is of the form (\ref{singul2}), whence,
by Proposition \ref{param}, $\frac{1}{(z+i) E} =
\frac{\phi}{z+i} \in H^2$. So the
spectrum $\{t_n\}$ is removable by Theorem \ref{annih}.
\end{proof}

A somewhat unexpected consequence of Theorem \ref{annih2} is that
adding a finite number of points to the spectrum helps it to
become removable, while deleting  a finite number of points may
make it nonremovable.

\begin{corollary}
\label{remov}
(i) If $\{t_n\}$ is removable, then for any finite set $\{\tilde
t_m\}_{m=1}^M$ disjoint with $\{t_n\}$
the spectrum $\{t_n\} \cup\{\tilde t_m\}_{m=1}^M$ is removable.

(ii) If the spectrum $\{t_n\}$ is removable, then by deleting a
finite number of elements of this sequence and adding the same
number of other elements we will always obtain a removable
spectrum.
\end{corollary}

\begin{proof}[Proof]
The statements follows immediately from Theorem \ref{annih2} since the
multiplication by a polynomial maps the Krein class
$\mathcal{K}_1$ into itself.
\end{proof}

\begin{corollary}
\label{nonremov} There exists a removable spectrum $\{t_n\}$, such
that $\{t_n\}_{n\ne m}$ is nonremovable for any $m$.
\end{corollary}

\begin{proof}[Proof]
Clearly, the spectrum $\{t_n\}_{n\in \mathbb{Z}}$, where $t_n = n$ for $n\in\BZ \setminus \{0\}$,
and $t_0$ is any real noninteger number is removable (take $A(z) = \frac{(z-t_0)  \sin \pi z}{z}$).
Now consider the spectrum $\{t_n\} = \{n\}_{n\in \mathbb{Z} \setminus\{0\}}$.
The corresponding generating function is $A(z) = \frac{\sin\pi z}{z}$ and $|F'(n)|
\asymp |n|^{-1}$. Hence the series $\sum_{n\ne 0}
\frac{1}{n^2|A'(t_n)|}$ diverges. Thus, $A\notin \mathcal{K}_1$
and so the spectrum is nonremovable.
\end{proof}
\medskip


\section{Examples of removable and nonremovable spectra}
\label{examples_rem}

In this subsection, we give some examples of removable and
nonremovable spectra with power growth
(one-sided and two-sided). To analyze the behavior of $|A'(t_n)|$
for the power growth of zeros we will use the Levin--Pfluger theory of
functions of completely regular growth \cite[Chapter 2]{levin}.
Assume that $t_n$ are all situated on the ray $\RR_+$ and
the counting function $n(r) = \#\{n: t_n\in [0,r]\}$ satisfies for
some $\rho \in(0,1)$,
\beqn
\label{lp1}
\lim\limits_{r\to\infty}
\frac{n(r)}{r^\rho} =D \in (0,\infty).
\neqn
Assume also that, for some $d>0$
\beqn
\label{lp2}
t_{n+1} -t_n \ge d\, t_n^{1-\rho}.
\neqn
Consider the discs $B_n = \{z: |z-t_n| <d\, t_n^{1-\rho}/2\}$.
Then we have
\beqn
\label{lp3}
\lim\limits_{t\to+\infty, \,
t\notin \cup_n B_n} \frac{\log|A(t)|}{t^{\rho}}  =  \pi D
%
%
\cot \pi\rho
\neqn
and
\beqn \label{lp4} \lim\limits_{t\to -\infty}
\frac{\log|A(t)|} {|t|^{\rho}} = \frac{\pi D}{\sin \pi\rho}.
\neqn
Moreover, it follows that \beqn \label{lp5}
\frac{\log|A'(t_n)|}{t_n^{\rho}} \to \pi D\, \cot\pi\rho, \qquad
n\to +\infty. \neqn

\begin{example}
1. {\it Two-sided symmetric power growth.} Assume that for some $\rho \in
(0,1)$, the spectrum $\{t_n\}$ satisfies
$$
\lim_{r\to\infty} \frac{\#\{t_n\in (-r,0)\}} {r^\rho}
 = \lim_{r\to\infty} \frac{\#\{t_n\in (0, r)\}} {r^\rho}  = D \in (0,\infty),
$$
and $t_{n+1} - t_n \ge d |t_n|^{1-\rho}$, $d>0$. Then
$t_n\asymp C |n|^{\frac 1 \rho}$ as $|n|\to \infty$.
It follows
from (\ref{lp3})--(\ref{lp5}) that
$$
\log|A'(t_n)| \sim \pi D |t_n|^{\rho} \,\cot\frac{\pi \rho}{2}
\asymp |n|, \qquad |n|\to+\infty,
$$
and so, by Lemma \ref{class-k1}, $A\in \mathcal{K}_1$ and the
spectrum $t_n$ is removable.

In particular, the spectrum $t_n = |n|^\gamma {\rm sign}\, n$,
$n\in \BZ \setminus \{0\}$,
where $t_0$ is any nonzero number in $(0,1)$,
is removable for any $\gamma>1$ (and for $\gamma =1$).
Note also that if $\gamma<1$, then the spectrum $\{t_n\}$ is not a zero set of
a function of exponential type and, hence, is
nonremovable.
\smallskip

2. {\it One-sided power growth.} Now let $t_n\in \RR_+$
satisfy conditions (\ref{lp1})--(\ref{lp2}). It follows from
(\ref{lp5}) and Lemma \ref{class-k1} that $\log |A'(t_n)| \asymp
|t_n|^\rho$ and the spectrum $\{t_n\}$ is removable when
$\rho<1/2$, while for  $\rho \in (1/2,1)$ we have $\log |A'(t_n)|
\asymp -|t_n|^\rho$ and the spectrum $\{t_n\}$ is nonremovable.
In particular, the power spectra $t_n= n^\gamma$, $n\in
\mathbb{N}$, are removable for $\gamma>2$ and
nonremovable for $\gamma<2$.
\smallskip

3. {\it The limit case: square growth.} For one-sided power
distributed zeros the limit case is the growth $t_n= n^2$, $n\in
\mathbb{N}$. This situation is more subtle. In this case
$$
  A(z)= \prod\limits_{n\in\mathbb{N}}\left(1-\frac{z}{n^2}\right)
  = \frac{\sin (\pi\sqrt{z})}{\pi\sqrt{z}},
$$
and so $|A'(t_n)| = (2t_n)^{-1} = (2n^2)^{-1}$. Then the series
(\ref{krein2}) converges and, by Lemma \ref{class-k1}, the
spectrum is removable. However, if we consider the spectrum
$\{n^2\}_{n\ge 2}$, then the corresponding generating function
$A_1$ satisfies $|A_1'(t_n)| \asymp |t_n|^{-2}$, and the spectrum
is nonremovable.

4. {\it Two-sided nonsymmetric growth.}
More generally, suppose that
$$
\lim_{r\to+\infty} \frac{\#\{\pm t_n\in (0, r)\}} {r^{\rho_\pm} }  = D_\pm  \in (0,\infty),
$$
and $t_{n+1} - t_n \ge d |t_n|^{1-\rho_\pm}$ for $\pm n >0$, where
$\rho_\pm\in(0,1)$ and $d>0$.
Define $u_+$, $u_+$ by
$$
u_\pm= D_\pm \cot \pi\rho_\pm + \frac {D_\mp }{\sin \pi\rho_\mp}\, .
$$
Then the same arguments as above imply that the spectrum is removable
if both $u_-$ and $u_+$ are positive and is not removable if at least one of these numbers
is negative. In particular, if $\rho_-, \rho_+ < 1/2$, then the spectrum is removable.
\end{example}

\begin{remark}
1. The special role of the exponent $2$ in the power distributed
spectra is well known; it may be seen,  e.g., in the problems of
weighted polynomial approximation on discrete subsets of $\RR$ \cite{bs}.

2. In the study of power growth, the regularity of the sequence is
important. It is easy to see that for any $\gamma>2$ there exists
a subset of $\{n^\gamma\}_{n\in \mathbb{N}}$, which is
nonremovable (take the set $n^\gamma$, $n\in [m_k, m_k+l_k]$ for
appropriately chosen $m_k,\, l_k \to \infty$).
\end{remark}

\begin{example}
Let $a>0$ and consider two shifted progressions:
$$
   t_n=
  \begin{cases}
   n+a, & n\ge 0, \ n\in\mathbb{Z}, \\
   n+1 -a,    & n<0, \ n\in\mathbb{Z},
  \end{cases}
$$
that is, $\{t_n\} = \{\dots, -a-1, -a \} \cup \{a, a+1, \dots\}$.
Then
\[
A(z) = \prod_{n=0}^\infty \Big( 1-\frac{z^2}{(n+a)^2}\Big) =
\frac
{\Ga(a)^2}{\Ga(a+z)\Ga(a-z)}
=
\frac
{\Ga(a)^2}{\pi\Ga(a+z)}\, \sin \pi(a-z) \Ga(1-a-z)
\, .
\]
Therefore for positive $k\in\BZ$,
\[
|A'(k+a)|
=
\frac
{\Ga(a)^2}{\Ga(k+2a)}\,  \Ga(k+1)
\asymp \Ga(a)^2  k^{1-2a} \quad \text{as }k\to +\infty.
\]
Since $A$ is an even function,
the series $\sum_{k>0} \frac{1}{k^2 |A'(t_k)|}$ converges (and
the spectrum is removable) if and only if $a<1$.
\end{example}

\begin{example}
One more class of examples with a nonremovable spectrum can be
obtained if we take a sequence of pairs of close points. In this
case the spectrum is `almost multiple' and thus nonremovable.
Let $\{t_n\}$ be a separated sequence (i.e., $\inf_n (t_{n+1} - t_n)>0$)
and consider the set $\{t_n\} \cup \{t_n + \delta_n\}$, where
$\delta_n \to 0$. If $\delta_n$ are sufficiently small, then we
can achieve that $|A'(t_n)|$ be small and, thus, (\ref{krein2}) is
not satisfied.
\end{example}
\medskip


\section{Liv\v sic's theorem on dissipative Volterra operators
with one-dimensional imaginary part}
\label{livst}

In this section we use our model to prove the above-mentioned
theorem  of Liv\v sic; it says that
any {\it dissipative Volterra operator}, which is a
rank one perturbation of a selfadjoint operator,
is unitary equivalent to the integration operator
(\cite{liv} or \cite[Ch. I, Th. 8.1]{Gohb_Kr_Volterr})
\footnote{We express our gratitude to N. Nikolski who attracted
our attention to Liv\v sic's theorem and suggested to deduce it
using our methods.}. Namely, we show the following:

\begin{theorem}[Liv\v sic, \cite{liv}]
\label{livs}
Let $\LL_0=\A_0 +i\B_0$
be a dissipative Volterra operator $($in some Hilbert space $H$$)$
such that both $\A_0$ and $\B_0$
are selfadjoint and $\B_0$ is of rank one.
Then the spectrum of $\A_0$ is given by $s_n = c (n+1/2)^{-1}$,
$n\in \mathbb{Z}$, for some $c\in \RR$, $c\ne 0$.
\end{theorem}

From this, one may deduce that $\A_0$ is unitary
equivalent to the selfadjoint integral operator (having the same spectrum)
$$
(\tilde \A f) (x) = i\int_0^{2\pi c} f(t)\, {\rm sign}\, (x-t)\,dt, \qquad
f\in L^2(0,2\pi c),
$$
while $\LL_0$ is unitary equivalent to the integration operator
$(\tilde \LL f) (x) = 2i \int_0^{x} f(t)\,dt$.

Since $\B_0\ge 0$, we have $\B_0 x =  (x, b_0) b_0$
for some $b_0 \in H$.
By passing to the unbounded inverses, we obtain
(after an obvious unitary equivalence) a singular rank one perturbation
$\LL = \LL(\A, a, b, \deab)$
of the operator $\A$ of multiplication by the independent variable
in some space $L^2(\mu)$, where $\mu = \sum_n \mu_n \delta_{t_n}$, $t_n
= s_n^{-1}$. Moreover, in the case of the positive imaginary part,
we may assume that $\deab = -1$ and  $a = i b$.

Applying the functional
model from Section \ref{fmodel}, we construct a pair $(\Theta, \varphi)$
as in Theorem \ref{rank-one-model}. Let $E = A-iB \in HB$ be such that
$\Theta = E^*/E$ and let $g = \varphi E$. Then
by \eqref{singul2}, \eqref{singul1},
we have
$$
\begin{aligned}
\frac{B(z)}{A(z)} & = \delta+ \sum_n \bigg(\frac{1}{t_n-z}
- \frac{1}{t_n}\bigg) |b_n|^2 \mu_n, \\
\frac{g(z)}{A(z)} & = -1 +  i \sum_n \bigg(\frac{1}{t_n-z}
- \frac{1}{t_n}\bigg) |b_n|^2 \mu_n,
\end{aligned}
$$
whence $g = -A + i(B -\delta A)$.

Since $\LL$ (and, thus, the model operator $\T$)
is the inverse to a Volterra operator, the spectrum of $\T$ is
the point at infinity. By Lemma~\ref{eigs}, $g$ has no zeros in $\BC$. Also, by
Theorem \ref{annih}, the function $E$ is of Carthwright class,
and the same is true for $g$. We conclude that $g(z) = \exp(i \pi cz)$
for some real $c$. Thus,
$$
e^{i\pi cz} =  -A(z) + i\big(B(z) -\delta A(z)\big).
$$
The functions $A$ and $B$ are real on the real axis. Taking the real parts,
we get $A(z) = -\cos \pi cz$, and so $t_n = c^{-1} (n+1/2)$,
$n\in \mathbb{Z}$, as required.
\bigskip


\section{Volterra rank one perturbations generated by `smooth' vectors}
\label{final}

Let $\A_0$ be a compact selfadjoint operator with the simple
point spectrum $\{s_n\}$, that is, the operator of multiplication by $x$
in $L^2(\nu)$, where $\nu = \sum_n \nu_n \delta_{s_n}$.
In this section we show that, for a rank one perturbation
$\LL_0 = \A_0 +a b^*$, the property of being a Volterra
operator is compatible with a certain smoothness of the vectors $a$ and $b$.
On the other hand, recall that the classical completeness theorem of Macaev
\cite{Mats61} (see also~\cite[Chapter V]{Gohb_Krein}) states that
in the case when $a$ or $b$ is in the range of $\A_0$ (i.e., $a\in x L^2(\nu)$
or $b\in x L^2(\nu)$) and $\Ker \LL_0 = 0$, the perturbed operator $\LL_0$
has a complete set of eigenvectors (or root vectors
in case of multiple spectrum):

\begin{theorem_}
[Macaev, 1961]
If $\LL_0 =\A_0 (I+S)$, where $\A_0$, $S$ are compact operators on a Hilbert space,
$\A_0$ is selfadjoint and $S$ is in the Macaev ideal $\mathfrak {S}_\om$
\textup(i.e., its singular numbers $s_k$ satisfy $\sum_{k\ge 1} \frac {s_k} k <\infty$\textup)
and $\ker \A_0 = \ker(I+S)=0$, then $\LL_0$ and $\LL_0^*$ have complete
sets of eigenvectors.
\end{theorem_}

The following theorem shows that any weaker smoothness of $a$ and $b$
can be achieved for Volterra rank one perturbations.
A special case of this result was given in \cite[Theorem 0.6]{bar-yak}.

\begin{theorem}
\label{smooth}
Let $s_n \to 0$, $s_n \ne 0$, and assume that $\{t_n\}$, where $t_n =s_n^{-1}$,
is a removable spectrum. Let $A$ be the corresponding function in the
Krein class $\mathcal{K}_1$,
given by~\eqref{cart}. Assume that for some $\gaa \in (0,2)$
we have
$$
\sum_n \frac{1}{|t_n|^{\gaa} |A'(t_n)|} <\infty.
$$
Let $\A_0$ be a selfadjoint operator with the point spectrum $\{s_n\}$
and trivial kernel \textup(i.e., $\A_0$ is the operator of multiplication by $x$
in $L^2(\nu)$ where $\nu = \sum_n \nu_n \delta_{s_n}$\textup).
Then for any $\alpha_1, \alpha_2 \in [0,1)$ with $\alpha_1+ \alpha_2 \le 2- \gaa$
there exist $a \in |x|^{\alpha_1} L^2(\nu)$ and $b \in |x|^{\alpha_2} L^2(\nu)$
such that $a, b \notin L^2(\nu)$ and
the spectrum of the perturbed operator $\LL_0 = \A_0 + ab^*$
equals~$\{0\}$.
\end{theorem}

\begin{proof}[Proof]
Let us pass to the equivalent problem for a singular
perturbation of an unbounded operator
$\A$ (which is unitary equivalent to $\A_0^{-1}$)
on $L^2(\mu)$, where $\mu=\sum_n\mu_n\delta_{t_n}$,
$t_n =s_n^{-1}$, and  $a'_n = (\A^{-1}a)_n = a_n/s_n$, $b'_n= b_n/s_n$.
Thus, for any $\alpha_1$ and $\alpha_2$ as above, we need to find
$\mu=\sum_n \mu_n \delta_{t_n}$, $\deab\in \RR$,
and $a', b' \notin L^2(\mu)$ such that
\beqn
\label{shar0}
\sum_n |a_n'|^2 |t_n|^{2\alpha_1-2}\mu_n<\infty, \qquad
\sum_n |b_n'|^2 |t_n|^{2\alpha_2-2}\mu_n<\infty
\neqn
(note that $a_n' = a_n t_n$) and the function
\beqn
\label{shar2}
\phi(z) =
\frac{1+ \Theta(z)}{2} \cdot
\bigg(\deab + \sum_n \Big(\frac{1}{t_n-z} -\frac{1}{t_n}\Big)
a_n' \overline{b_n'} \mu_n \bigg)
\neqn
has no zeros in $\BC$.

For a function $A\in \mathcal{K}_1$ we have
\beqn
\label{shar1}
\frac{1}{A(z)} = q + \sum_n c_n\bigg(\frac{1}{t_n-z} -\frac{1}{t_n}\bigg),
\neqn
where $q = 1/A(0)$, $c_n = - 1/A'(t_n)$ and $\sum t_n^{-2}|c_n| < \infty$.
We represent $c_n$ as $c_n=a_n' \overline{b_n'}\, \mu_n$, where
$a_n'$ and $b_n'$ have the required properties.
Once such $a'$ and $b'$ have been constructed, we define the function $\Theta$
by the formulas \eqref{singul2} and \eqref{the} with $b_n'$ in place of
$b_n$, that is, we put
$$
i\frac{1-\Theta(z)}{1+\Theta(z)} =
\sum_n \bigg(\frac{1}{t_n-z}-\frac{1}{t_n}\bigg) |b'_n|^2 \mu_n.
$$
Next, define $\phi$ by (\ref{shar2}), with
$a_n' \overline{b_n'} \mu_n = c_n$ and with
$q$ in place of $\deab$.
If we now put $E = \frac{2A}{1+\Theta}$, then, clearly, $E$ is an entire function,
$\phi = 1/E$ by (\ref{shar1}),
and $\Theta = E^*/E$ (see (\ref{shar3})) whence $E\in HB$.
Thus, the function $\phi$ has no zeros.

Let $\{\mu_n\}$ be an arbitrary sequence of positive numbers. Put
$$
a_n' = \frac{|c_n|^{1/2} |t_n|^{(2-2\alpha_1-\gaa)/2}}{\mu_n^{1/2}},
\qquad
b_n' = \frac{|c_n|^{1/2} |t_n|^{(2\alpha_1 +\gaa-2)/2}}{\mu_n^{1/2}}.
$$
Then
$$
\sum_n |a_n'|^2 |t_n|^{2\alpha_1-2}\mu_n =
\sum_n \frac{|c_n|}{|t_n|^{\gaa}}<\infty,
$$
$$
\sum_n |b_n'|^2 |t_n|^{2\alpha_2-2}\mu_n = \sum_n \frac{|c_n|}
{|t_n|^{4-2\alpha_1- 2\alpha_2 -\gaa}}<\infty,
$$
since $\alpha_1+\alpha_2 \le 2-\gaa$.

If $a'$ and $b'$ are not in $L^2 (\mu)$, then the theorem is proved.
Otherwise, choose a sequence $p_{2n+1}\ge 1$ such that
$$
\sum_n p_{2n+1}^2 |a_{2n+1}'|^2 |t_{2n+1}|^{2\alpha_1-2}\mu_{2n+1} <\infty,
\qquad
\sum_n p_{2n+1}^2 |a_{2n+1}'|^2 \mu_{2n+1} = \infty
$$
(this is, obviously, possible, because $|t_n| \to \infty$ and $\alpha_1 <1$).
Analogously, we choose $p_{2n}\le 1$ so that
$$
\sum_n p_{2n}^{-2} |b_{2n}'|^2 |t_{2n}|^{2\alpha_2-2}\mu_{2n} <\infty,
\qquad
\sum_n p_{2n}^{-2} |b_{2n}'|^2 \mu_{2n} = \infty.
$$
Then, clearly, $\tilde a_n = p_n a_n'$ and $\tilde b_n = p_n^{-1} b_n'$ are not in $L^2(\mu)$,
$\tilde a \in x^{\alpha_1-1} L^2(\mu)$,
$\tilde b \in x^{\alpha_2-1} L^2(\mu)$
and $\tilde a_n \overline{\tilde b_n} = c_n$.
\end{proof}

\begin{remark}
One can compare Theorem~\ref{smooth} with Macaev's theorem
mentioned above as well as with the following result from \cite{bar-yak}
(Theorem 3.3, Statement (2)):
{\it Let $\A$ be an unbounded cyclic selfadjoint operator with
discrete spectrum $\{t_n\}$, $t_n \ne 0$, and let the data $(a,b, \deab)$ satisfy
\begin{align}
\label{smooth-sing}
& \sum_n \frac{|a_n b_n| \mu_n}{|t_n|}<\infty, \\
\label{import}
& \sum_n \frac{a_n\overline b_n \mu_n}{t_n} \ne \deab.
\end{align}
Then the singular rank one perturbation $\LL = \LL(\A, a, b, \deab)$
and its adjoint $\LL^*$ have complete sets of eigenvectors. }

Note that \eqref{smooth-sing} is satisfied if $1 \le \al_1+\al_2$.
However, there is no contradiction with Theorem \ref{smooth}. Indeed, looking at the
the asymptotics when $y\to\infty$ in
$$
\frac{1}{A(iy)} = \deab -\sum_n \frac{c_n}{t_n} +\sum_n \frac{c_n}{t_n -iy},
$$
where $c_n = a_n\overline b_n \mu_n$, we see that  \eqref{import} is not satisfied for the perturbation
constructed in Theorem \ref{smooth}.
\end{remark}

Thus, in contrast to Macaev's theorem (which applies to
the so-called weak perturbations of the form $\A_0(I+S)$ or $(I+S)\A_0$),
we have the following corollary of Theorem \ref{smooth}:

\begin{corollary}
\label{contr}
For any $\alpha_1, \alpha_2 \in (0,1)$, $\alpha_1+\alpha_2 >1$, 
there exist a positive compact operator $\A_0$ and its rank one 
perturbation $\LL_0$ of the form
$$
\LL_0 = \A_0 + \A_0^{\alpha_1} S \A_0^{\alpha_2},
$$
where $S$ is a rank one operator and $\Ker \LL_0 = \Ker \LL_0^* = 0$,
such that $\LL_0$ is a Volterra operator. 
\end{corollary}



\begin{thebibliography}{25}

\bibitem{ac70} P.R. Ahern, D.N. Clark, Radial limits and
invariant subspaces, {\it Amer. J. Math.} {\bf 92} (1970),
332--342.

\bibitem{alb-kur-2000}
S. Albeverio, P. Kurasov,
{\it Singular Perturbations of Differential Operators:
Solvable Schr\"odinger Type Operators}, London Math. Society Lecture Note Series, 
Vol. 271,  Cambridge Univ. Press, Cambridge, 2000.

\bibitem{alb-kokoshm05}
S. Albeverio, A. Konstantinov, V. Koshmanenko,
Decompositions of singular continuous spectra of ${\mathcal H}_{-2}$-class rank one perturbations.
{\it Integr. Equat. Oper. Th.} {\bf 52} (2005), 4, 455--464.

\bibitem{bar-arkiv} A.D. Baranov, Polynomials in the de\,Branges spaces of entire
functions, {\it Ark. Mat.} {\bf 44} (2006), 1, 16--38.

\bibitem{bar-yak} A.~Baranov, D.~Yakubovich,
Completeness and spectral synthesis
of nonselfadjoint one-dimensional perturbations
of selfadjoint operators, arXiv:1212.5965.

\bibitem{Biy_Dzhumb}
B.N. Biyarov, S.A. Dzhumabaev,
A criterion for the Volterra property of boundary
value problems for Sturm--Liouville equations,
{\it Math. Notes} {\bf 56} (1994), 1, 751--753.

\bibitem{bs}
A. Borichev, M. Sodin, The Hamburger moment problem
and weighted polynomial approximation on discrete
subsets of the real line, {\it J. Anal. Math.} {\bf 76} (1998), 219--264.

\bibitem{br} L. de Branges, {\it Hilbert Spaces of Entire Functions},
Prentice Hall, Englewood Cliffs (NJ), 1968.

\bibitem{cl} D.N. Clark, One-dimensional perturbations of restricted shifts,
{\it J. Anal. Math.} {\bf 25} (1972), 169--191.

\bibitem{simon-makarov} R. Del Rio, N. Makarov, B. Simon,
Operators with singular continuous spectrum. II. Rank one operators,
{\it Comm. Math. Phys.} {\bf 165} (1994), 1, 59--67.

\bibitem{Dunf_Schwarz}
N. Dunford, J.T. Schwartz, {\it Linear Operators, Vol. 2: Spectral Theory},
Interscience, New York, 1963.

\bibitem{Gohb_Krein}
I. Gohberg, M. Krein, {\it Introduction to the Theory of
Linear Nonselfadjoint Operators}, Amer. Math. Soc., Providence, RI, 1969.

\bibitem{Gohb_Kr_Volterr}
I. Gohberg, M. Krein, {\it Theory and Applications of Volterra Operators
in Hilbert Space},  Amer. Math. Soc., Providence, RI, 1970.

\bibitem{Gubr-Tar2010}
G.M. Gubreev, A.A. Tarasenko, Spectral decomposition of model
operators in de Branges spaces,
{\it Mat. Sb.} 201 (2010), 11, 41--76;
English transl.: {\it Sb. Math.} 201 (2010), 11, 1599--1634.

\bibitem{Gubr-Lat2011}
G.M. Gubreev, Yu.D. Latushkin,
Functional models of nonselfadjoint operators, strongly continuous semigroups,
and matrix Muckenhoupt weights, {\it Izv. Ross. Akad. Nauk Ser. Mat.} {\bf 75} 
(2011), 2, 69--126;
English transl.: {\it Izv. Math.} {\bf 75} (2011), 2, 287--346.

\bibitem{hj}
V. Havin, B. J\"oricke, \textit{The Uncertainty Principle in Harmonic Analysis},
Springer-Verlag, Berlin, 1994.

\bibitem{hm1} V.P. Havin,  J. Mashreghi,
Admissible majorants for model subspaces of $H^2$. Part I: slow
winding of the generating inner function;
Part II: fast winding of the generating inner function
{\it Can. J. Math.} {\bf 55} (2003), 6, 1231--1263; 1264--1301.

\bibitem{kw}
M. Kaltenb\"ack, H. Woracek,
Hermite--Biehler functions with zeros close to the imaginary
axis, {\it Proc. Amer. Math. Soc.} {\bf 133} (2005), 1, 245--255.

\bibitem{Kap}
V.V. Kapustin, One-dimensional perturbations of singular unitary operators,
{\it Zapiski Nauchn. Sem. POMI} {\bf 232} (1996),
118--122;
English transl. in {\it J. Math. Sci. (New York)} {\bf 92} (1998), 1,
3619--3621.

\bibitem{Khromov}
A.P. Khromov,
Finite-dimensional perturbations of Volterra operators,
\textit{Sovrem. Mat. Fundam. Napravl.} {\bf 10} (2004), 3--163;
English transl.: \textit{J. Math. Sci.} (N.Y.)
{\bf 138} (2006), 5, 5893--6066.

\bibitem{ko0} P. Koosis, \textit{Introduction to $H^p$ spaces},
Cambridge Univ. Press, Cambridge, 1980.

\bibitem{ko1} P. Koosis,
\textit{The Logarithmic Integral I}, Cambridge Stud. Adv. Math.
{\bf 12}, 1988.

\bibitem{krein47}
M. Krein, A contribution to the theory of entire functions
of exponential type, {\it Izv. Akad. Nauk SSSR Ser. Mat.} {\bf 11}
(1947), 4, 309--326.



\bibitem{levin}
B.Ya. Levin, \textit{Distribution of Zeros
of Entire Functions}, GITTL, Moscow, 1956; English transl.: Amer.
Math. Soc., Providence, 1964; revised edition: Amer. Math. Soc.,
1980.


\bibitem{treil}
C. Liaw, S. Treil,
Rank one perturbations and singular
integral operators, {\it J. Funct. Anal.} {\bf 257} (2009), 6, 1947--1975.

\bibitem{liv}
M.S. Liv\v sic, On spectral decomposition of linear
nonselfadjoint operators, {\it Mat. Sb.}, {\bf 34(76)} (1954), 1, 145--199.

\bibitem{Lyants-Stor-book}                            
V.\`{E}. Lyantse, O.G. Storozh, {\it Methods of the Theory of Unbounded Operators},
Naukova Dumka, Kiev, 1983 (in Russian).

\bibitem{LukGub}
G. V. Lukashenko, G. M. Gubreev,
About nilpotent $C_0$-semigroups of operators
in the Hilbert spaces and criteria for similarity to the integration operator.
{\it Methods Funct. Anal. Topology} {\bf 14} (2008), 1, 60--66.

\bibitem{malam}
M.M. Malamud,
Remarks on the spectrum of one-dimensional
perturbations of Volterra operators, {\it Mat. Fiz.} {\bf 32} (1982), 99--105.


\bibitem{Mats61}
V.I. Macaev,
A class of completely continuous operators,
{\it Dokl. Akad. Nauk SSSR} {\bf 139} (1961), 3, 548--551;
English transl.: {\it Soviet Math. Dokl.} {\bf 2} (1961), 972--975.


\bibitem{MatsSod}
V. Macaev, M. Sodin, Entire functions and compact operators with
$S_p$-imaginary component.
Entire functions in modern analysis (Tel-Aviv, 1997), 243--260,
Israel Math. Conf. Proc., 15, Bar-Ilan Univ., Ramat Gan, 2001.

\bibitem{Naim-book}
M. A. Na\v \i mark,
\margp{\bf TO CHECK}
{\it Linear Differential Operators. Part II: Linear 
Differential Operators in Hilbert space}, 
Frederick Ungar Publishing Co., New York, 1968.

\bibitem{nk12}
N.K. Nikolski, \textit{Operators, Functions, and Systems: an
Easy Reading. Vol. 1-2}, Math. Surveys Monogr., Vol. 92--93, AMS,
Providence, RI, 2002.

\bibitem{Posilicano}
A. Posilicano, Self-adjoint extensions of restrictions,
\textit{Oper. Matrices} {\bf 2} (2008), 4, 483--506.

\bibitem{Romaschenko}
G.S. Romaschenko, On spectra of one-dimensional
perturbations of Volterra operators in Sobolev spaces,
{\it Ukrainian Math. Bull.}, {\bf 4} (2007), 3, 437--451.

\bibitem{Ryzhov}
V. Ryzhov,  Functional model of a class of nonselfadjoint
extensions of symmetric operators,
{\it Operator Theory: Advances and Applications}, Vol. 174, 117--158.

\bibitem{simon}
B. Simon, Spectral analysis of rank one perturbations and applications,
{\it CRM Lecture Notes}, Vol. 8 (J. Feldman, R. Froese, L. Rosen, eds.),
Amer. Math. Soc., Providence, RI, 1995, 109--149.

\bibitem{st1}
L.O. Silva, J.H. Toloza,
On the spectral characterization of entire operators with deficiency indices (1,1),
{\it J. Math. Anal. Appl.}, {\bf 367} (2010), 2, 360--373.

\bibitem{st2}
L.O. Silva, J.H. Toloza,
The class of $N$-entire operators, arXiv:1208.2218.

\bibitem{wor}
H. Woracek, De~Branges spaces of entire functions closed under
forming difference quotients, {\it Integr. Equat. Oper. Th.}
{\bf 37} (2000), 2, 238--249.

\end{thebibliography}
\end{document}